\documentclass[11pt]{amsart}  %@@
\usepackage{amssymb, eucal}
\usepackage[all,arc]{xy}
\usepackage{enumerate}
\usepackage{color}

\newenvironment{tfae}{
\begin{enumerate}}{\end{enumerate}}%\def\theenumi{\arabic{enumi}}

\newtheorem{prop}{Proposition}[section]
\newtheorem{lemma}[prop]{Lemma}
\newtheorem{theorem}[prop]{Theorem}
\newtheorem{corollary}[prop]{Corollary}
\theoremstyle{definition}
\newtheorem{definition}[prop]{Definition}

\newtheorem{remark}[prop]{Remark}

\renewcommand{\subsection}[1]{\addtocounter{subsection}{1}
\vspace*{2ex}\noindent\textbf{\thesubsection}\hspace{1ex}{\bf #1}}

\newdir{ >}{{}*!/-8pt/@{>}}

\def\mathrmdef#1{\expandafter\def\csname#1\endcsname{{\rm#1}}}
\mathrmdef{can}\mathrmdef{ev}\mathrmdef{coker}\mathrmdef{ker}\mathrmdef{id}\mathrmdef{Id}\mathrmdef{lex}\mathrmdef{op}\mathrmdef{pr}
\mathrmdef{prod}\mathrmdef{Aut}\mathrmdef{AUT}\mathrmdef{Pt}

\def\mathsfdef#1{\expandafter\def\csname#1\endcsname{{\sf#1}}}
\mathsfdef{F}
\mathsfdef{P}
\mathsfdef{L}
\mathsfdef{Gp}
\mathsfdef{oGp}
\mathsfdef{R}
\mathsfdef{U}
\mathsfdef{Rali}
\mathsfdef{Strong}

 \def\mathbfdef#1{\expandafter\def\csname#1\endcsname{{\rm\bf#1}}}

\mathsfdef{Grp}
\mathsfdef{Mon}
\mathsfdef{Ord}
\mathsfdef{Set}
\mathsfdef{OrdGrp}
\mathsfdef{TopGrp}

\def\NN{\mathbb{N}}

\def\ZZ{\mathbb{Z}}
\def\QQ{\mathbb{Q}}

\def\CC{\mathsf{C}}

\def\SS{\mathcal{S}}
\def\calP{\mathcal{P}}

\counterwithin*{equation}{section}

\begin{document}  %@@
\title{On split extensions of preordered groups}  %@@
\author{Maria Manuel Clementino}
\address{University of Coimbra, CMUC, Department of Mathematics, 3000-143 Coimbra, Portugal}\thanks{The first author acknowledges partial financial assistance by the Centre for Mathematics
of the University of Coimbra -- UIDB/00324/2020, funded by the Portuguese Government through
FCT/MCTES}
\email{mmc@mat.uc.pt}

\author{Carla Ruivo}\address{Department of Mathematics, 3000-143 Coimbra, Portugal}\thanks{This work was carried out within the scope of a grant from the Gulbenkian Foundation programme \emph{Novos Talentos em Matem\'{a}tica}}
\email{cruivo@proton.me}

\begin{abstract}
We investigate the behaviour of split extensions in the category $\OrdGrp$ of (pre)\-ordered groups. Namely we show that the lexicographic order plays a key role on the existence of compatible orders for semidirect products, establishing necessary and sufficient conditions for such existence; we prove that the Split Short Five Lemma holds for stably strong split extensions, and identify classes of split extensions which admit a classifier.
\end{abstract}
\subjclass[2020]{06F15, 18E13, 08C05}
%06F15 Ordered groups; 18E13 Protomodular cats, etc; 08C05 Cats of algebras
\keywords{preordered group, positive cone, split extension, $S$-protomodular category}

\maketitle  %@@

\section{Introduction}

In \cite{CMFM} the authors studied the behaviour of the category $\OrdGrp$ of preordered groups and monotone group homomorphisms. They show in particular that, unlike the categories $\Grp$ of groups and $\TopGrp$ of topological groups, $\OrdGrp$ is not protomodular, and, consequently, the Split Short Five Lemma does not hold. This relies essentially on the study of possible orders (as very common in the literature, throughout by \emph{order} we mean preorder) in a semidirect product $X\rtimes_\varphi B$ in $\Grp$ of two ordered groups $X$ and $B$ so that
\[\xymatrix{X\ar[r]^-{\langle 1,0\rangle}&X\rtimes_\varphi B\ar@<-2pt>[r]_-{\pi_2}&B\ar@<-2pt>[l]_-{\langle 0,1\rangle}}\]
is a split extension in $\OrdGrp$. Calling these orders \emph{compatible}, it is shown in \cite{CMFM} that compatible orders must contain the product order and be contained in the (reverse) \emph{lexicographic order}, and that they may not exist, or there may be plenty of them.

This note complements the study of split extensions presented in \cite{CMFM}. Indeed, we establish necessary and sufficient conditions in order to a compatible order in $X\rtimes_\varphi B$ exist, and show that the compatibility of the lexicographic order is essential. In case it exists, we identify both the maximal order, which is the lexicographic order, and the minimal one, which in general does not coincide with the product order.

Strong split extensions, or strong points, are exactly those with minimal order. It is shown that strong points need not be stable under pullback, and that the Split Short Five Lemma holds for stably strong points. This way we identify a sort of relative protomodularity which is possibly weaker than the notion introduced in \cite{BMMS}, but which still guarantees the validity of the Split Short Five Lemma, and, consequently, the reflection of isomorphisms by the corresponding change of base functors, as for protomodularity \cite{B,BBbook}.

Finally, we investigate the existence of split extension classifier, showing that this is possible only for special classes of split extensions, as for instance for those that may be identified as ralis (=right adjoints and left inverses) for the $\Ord$-enrichment of $\OrdGrp$ considered in \cite{CMR}.

\section{Preliminaries} \label{section first properties}

Let $\OrdGrp$ be the category of ordered groups and monotone group homomorphisms. By an ordered group we mean a (non-necessarily abelian) group $X$ equipped with an \emph{order} (i.e. a reflexive and transitive) relation $\leq$  such that the group operation (here denoted by $+$) is monotone. We point out that in general the group inversion is not monotone; it is in fact necessarily anti-monotone. The order is completely determined by its \emph{positive cone} $P=\{x\in X\,;\,x\geq 0\}$, which is a submonoid of $X$ closed under conjugation. Moreover, for any group $X$, any submonoid closed under conjugation defines an order on $X$.

\begin{remark}
Given a subset $A$ of a group $X$, the least order on $X$ whose positive cone contains $A$ is obtained in two steps: first we consider the closure $\hat{A}$ of $A$ under conjugation and then the closure of $\hat{A}$ under addition, which we denote by $\langle A\rangle$.
\end{remark}

The category $\OrdGrp$ has both an algebraic and a topological flavour; indeed, it was shown in \cite{CMFM} that the forgetful functors $\OrdGrp\to\Grp$ and $\OrdGrp\to\Ord$ are, respectively, topological and monadic. These functors allow us to construct limits and colimits easily, and are the basis for the categorical study of $\OrdGrp$.
Therefore, in order to study the behaviour of split extensions in $\OrdGrp$ we start by recalling briefly the behaviour of split extensions in $\Grp$. By \emph{split extension} we mean a short exact sequence $(\xymatrix{X\ar[r]^k&A\ar[r]^f&B})$, with $k=\ker f$ and $f=\coker k$, where $f$ is a split epimorphism and a splitting of $f$ is given:
\begin{equation}\label{eq:splitgrp}
\xymatrix{X\ar[r]^k&A\ar@<-2pt>[r]_f&B\ar@<-2pt>[l]_s}
\end{equation}
A morphism between split extensions is a triple $(a,b,c)$ making the following diagram commutative
\begin{equation}\label{eq:morph}
\xymatrix{X\ar[d]_{a}\ar[r]^k&A\ar[d]^{b}\ar@<-2pt>[r]_f&B\ar[d]^{c}\ar@<-2pt>[l]_s\\
X'\ar[r]^{k'}&A'\ar@<-2pt>[r]_{f'}&B'\ar@<-2pt>[l]_{s'}}
\end{equation}
so that $k'\cdot a=b\cdot k$, $f'\cdot b=c\cdot f$, and $s'\cdot c=b\cdot s$.

It is well-known that every split extension in $\Grp$ is isomorphic to one given by a \emph{semidirect product}, i.e. $A$ is necessarily isomorphic to the group $X\rtimes_\varphi B$ having as underlying set the cartesian product $X\times B$ and, for $(x,b),(x',b')$ in $X\times B$,
\[(x,b)+(x',b')=(x+\varphi(b,x'),b+b'),\]
where $\varphi\colon B\times X\to X$ is an \emph{action} of $B$ on $X$ (so that $\varphi(0,x)=x$, $\varphi(b,x+x')=\varphi(b,x)+\varphi(b,x')$, $\varphi(b',\varphi(b,x))=\varphi(b'+b,x)$). This induces an isomorphism of split extensions
\[\xymatrix{X\ar@{=}[d]\ar[r]^k&A\ar[d]^{\theta}\ar@<-2pt>[r]_f&B\ar@{=}[d]\ar@<-2pt>[l]_s\\
X\ar[r]^-{\langle 1,0\rangle}&X\rtimes_\varphi B\ar@<-2pt>[r]_-{\pi_2}&B\ar@<-2pt>[l]_-{\langle 0,1\rangle}},\]
with $\theta(a)=(a-sf(a),f(a))$ for every $a\in A$, and $\varphi_b(x)=\varphi(b,x)=s(b)+k(x)$ (see for instance \cite[Section 4.1]{Cl21} for details).

In $\Grp$ split extensions with given kernel have a \emph{classifier}, in the sense that the category of split extensions with kernel $X$ has a terminal object, i.e. there exists a split extension with kernel $X$
\begin{equation}\label{eq:autX}
\xymatrix{X\ar[r]&X\rtimes\AUT(X)\ar@<-2pt>[r]&\AUT(X)\ar@<-2pt>[l]}
\end{equation}
such that, for each split extension \eqref{eq:splitgrp} there exists exactly one morphism $(a,b,c)$ from \eqref{eq:splitgrp} to \eqref{eq:autX} with $a=\id_X$; here $\AUT(X)$ is the group $\{\alpha\colon X\to X\,;\,\alpha\mbox{ is an automorphism}\}$, with the operation given by composition, and the addition on the semidirect product $X\rtimes\AUT(X)$ given by
\[(x,\alpha)+(x',\alpha')=(x+\alpha(x'),\alpha\cdot\alpha').\]
The claimed morphism of split extensions is then given by
\begin{equation}\label{eq:AutX}
\xymatrix{X\ar@{=}[d]\ar[r]^k&A\ar[d]^{\theta}\ar@<-2pt>[r]_f&B\ar@{=}[d]\ar@<-2pt>[l]_s\\
X\ar@{=}[d]\ar[r]^-{\langle 1,0\rangle}&X\rtimes_\varphi B\ar@<-2pt>[r]_-{\pi_2}\ar[d]^{1\times\overline{\varphi}}&B\ar@<-2pt>[l]_-{\langle 0,1\rangle}\ar[d]^{\overline{\varphi}}\\
X\ar[r]^-{\langle 1,0\rangle}&X\rtimes\AUT(X)\ar@<-2pt>[r]_-{\pi_2}&\AUT(X)\ar@<-2pt>[l]_-{\langle 0,1\rangle}}\end{equation}
where $\overline{\varphi}(b)=\varphi_b$.

Alternatively one also says that in $\Grp$ \emph{actions are representable}, since this property can be stated as \emph{representability of a functor} into $\Set$ (see \cite{BJK1, BJK2} for details).

\section{Compatible orders} \label{section split extensions}

Throughout this section
\begin{equation}\label{eq:split}
\xymatrix{(X,P_X) \ar[r]^-{\langle 1,0\rangle} & X\rtimes_\varphi B \ar@<-2pt>[r]_-{\pi_B} & (B,P_B) \ar@<-2pt>[l]_-{\langle 0,1\rangle}}
\end{equation} is a split extension in $\Grp$, and $X$ and $B$ are ordered groups, with positive cones $P_X$ and $P_B$ respectively.
As shown in \cite{CMFM} there may be no -- or there may be plenty of -- orders in $X\rtimes_\varphi B$ making \eqref{eq:split} a split extension in $\Ord\Grp$. Here we will identify exactly those split extensions \eqref{eq:split} for which there exist compatible orders. For that we make use of:
\begin{itemize}
\item the \emph{product order}, with positive cone $P_\prod=P_X\times P_B$, and
\item the \emph{lexicographic order}, with positive cone
\[P_\lex=\{(x,b)\in X\times B\,|\, b>0\mbox{ or }(b\sim 0\mbox{ and }x\geq 0)\}.\]
\end{itemize}
We point out that these orders need not be compatible in \eqref{eq:split}, and that this lexicographic order is larger than the one defined in \cite{CMFM}. In fact the lexicographic order in Proposition 5.1 of \cite{CMFM} should have been this one and not the one considered in \cite{CMFM}.

\begin{prop}
For a positive cone $P$ in $X \rtimes_{\varphi} B$, the following conditions are equivalent:
\begin{tfae}
\item $P$ is compatible in \eqref{eq:split};
\item $P_\prod \subseteq P \subseteq P_{\lex}$.
\end{tfae}
\end{prop}

\begin{proof} (i)$\Rightarrow$(ii):
Monotonicity of $\xymatrix{X \ar[r]^-{\langle 1,0\rangle} & X\rtimes_\varphi B & B \ar[l]_-{\langle 0,1\rangle}}$ implies $(x,0)\in P$ and $(0,b)\in P$ when $x\in P_X$ and $b\in P_B$, and therefore $(x,b)=(x,0)+(0,b)\in P$, that is $P_{\prod}\subseteq P$.

If $(x,b)\in P$ then necessarily $b\geq 0$ because $\pi_B$ is monotone. When $b\sim 0$ and $(x,b)\in P$, then $(0,b)\sim(0,0)$ and so $(x,0)=(x,b)-(0,b)\in P$, which implies $x\geq 0$ because $\langle 1,0\rangle\colon X\to X\rtimes_\varphi B$ is a kernel in $\Ord\Grp$.

(ii)$\Rightarrow$(i): For every $b\in P_B$, $x\in X$, $(x,b)\geq 0$ implies $b\geq 0$, hence $\pi_B$ is monotone. For every $x\in X$, $(x,0)\in P$ if and only if $x\in P_X$, and therefore $\langle 1,0\rangle\colon (X,P_X)\to(X\rtimes_\varphi B,P)$ is the kernel of $\pi_B$. Finally, if $b\geq 0$ then $(0,b)\in P_\prod\subseteq P$, and so $\langle 0,1\rangle\colon (B,P_B)\to(X\rtimes_\varphi B,P)$ is also monotone.
\end{proof}

The lexicographic order plays an essential role here, since it is compatible as soon as there is a compatible order, as we show next.

\begin{theorem}\label{th:lexic}
Given \eqref{eq:split}, the following conditions are equivalent:
\begin{tfae}
\item There is a compatible order in \eqref{eq:split}.
\item For every $b\in B$ $\varphi_b$ is monotone, and, if $b\sim 0$ then $\varphi_b\sim \id$ (pointwise).
\item The lexicographic order is compatible in \eqref{eq:split}.
\end{tfae}
\end{theorem}

\begin{proof}
(i)$\Rightarrow$(ii): Let $P$ be a compatible positive cone. If $x\geq 0$, then, for any $b\in B$, $(\varphi_b(x),0)=(0,b)+(x,0)-(0,b)\in P$, that is $\varphi_b(x)\geq 0$. Now let $b\sim 0$ in $B$. Then, for every $x\in X$, $(x-\varphi_b(x),b)=(x,0)+(0,b)-(x,0)\in P$ and so $x\geq\varphi_b(x)$; this, together with $-x\geq\varphi_b(-x)=-\varphi_b(x)$ gives $x\sim \varphi_b(x)$.

(ii)$\Rightarrow$(iii): We need to prove that $P_\lex$ is a positive cone, that is, it is closed under addition and conjugation. If $(x,b),(x',b')\in P$ and $b>0$ or $b'>0$, then, obviously, $(x,b)+(x',b')\in P$; if both $b\sim 0$ and $b'\sim 0$, then in $(x,b)+(x',b')=(x+\varphi_b(x'),b+b')$ we have $b+b'\sim 0$ and $x+\varphi_b(x')\geq 0$ because both $x$ and $\varphi_b(x')$ are positive. Now let $(x,b)\in P$ and $(y,a)\in X\rtimes_\varphi B$. Then
\[(y,a)+(x,b)-(y,a)=(y,a)+(x,b)-(0,a)-(y,0)=(y+\varphi_a(x)-\varphi_{a+b-a}(y),a+b-a);\]
if $b>0$, then $a+b-a>0$ and so the pair above belongs to $P_\lex$; if $b\sim 0$, then $x\geq 0$ and $a+b-a\sim 0$, and so $y+\varphi_a(x)-\varphi_{a+b-a}(y)\sim y+\varphi_a(x)-y\geq 0$.

(iii)$\Rightarrow$(i) is trivial.
\end{proof}

\begin{corollary}
If the order in $B$ is antisymmetric, then there is a compatible order in \eqref{eq:split} if, and only if, $\varphi_b$ is monotone for every $b\in B$.
\end{corollary}

We may consider now the set $\calP$ of compatible positive cones for \eqref{eq:split}. We have just shown that either $\calP$ is empty or it has a top element, $P_\lex$, when ordered by inclusion. A split extension \eqref{eq:split} where $X\rtimes_\varphi B$ has the lexicographic order will be called \emph{maximal}.

\begin{prop}
Either $\calP=\emptyset$ or $\calP$ is a complete lattice.
\end{prop}
\begin{proof}
It is easily checked that the meet of compatible orders is a compatible order.
\end{proof}

Hence, if $\calP\neq\emptyset$, there is a least compatible order, which we call \emph{minimal} and describe next.

\begin{prop}\label{prop:prod}
Let $\calP\neq\emptyset$.
\begin{enumerate}
\item[\em (1)] The positive cone of the least compatible order for \eqref{eq:split} is $\langle P_\prod\rangle$, that is, the one generated by $P_X\times P_B$.
\item[\em (2)] Moreover, it coincides with $P_\prod$ if, and only if, $\varphi_b\sim\id$ for every positive element $b$ of $B$.
\end{enumerate}
\end{prop}

\begin{proof}
(1) is obvious. To show (2) we use \cite[Proposition 5.2]{CMFM}, which assures that $P_\prod$ is compatible if and only if $\varphi_b(x)\geq x$ for all $b\in P_B$ and $x\in X$. But this, together with $\varphi_b(-x)\geq -x$ gives $\varphi_b\sim\id$ as claimed.
\end{proof}

\begin{remark}
On one hand there are examples of \eqref{eq:split} with no compatible order, like for instance
\[\xymatrix{ (\ZZ,\NN) \ar[r]^{\langle 1,0\rangle} & \ZZ\rtimes \ZZ \ar@<-2pt>[r]_{\pi_2} & (\ZZ,\ZZ) \ar@<-2pt>[l]_{\langle 0,1\rangle}}\]
with $\varphi_b(x)=(-1)^bx$, since $\varphi_b$ is not monotone. On the other hand, Example 5.8 of \cite{CMFM} shows that there may be plenty of possible positive cones, even in the case when both $X$ and $B$ are abelian and have antisymmetric orders: it is shown there that, for the split extension
\[\xymatrix{ (\ZZ,\NN) \ar[r]^{\langle 1,0\rangle} & \ZZ\times \ZZ \ar@<-2pt>[r]_{\pi_2} & (\ZZ,\NN) \ar@<-2pt>[l]_{\langle 0,1\rangle}}\]
$\calP$ is an uncountable set. Indeed, a positive cone $P$ on $\ZZ\times\ZZ$ can be determined by a family of sets $(X_j=\{n\in\ZZ\,;\,(n,j)\in P\})_{j\in\ZZ}$, with $X_j=\emptyset$ if $j<0$ and $X_j=\uparrow -x_n\subseteq \ZZ$ for $j\geq 0$ so that $x_0=0$ and $(x_n)_{n\in\NN}$ is a sequence in $\NN\cup\{\infty\}$ such that $x_{n+m}\geq x_n+x_m$.
\end{remark}
 Next we present a characterization of compatible orders inspired by the latter example. We say that family of subsets $(X_b)_{b\in B}$ of $X$ is compatible if $P=\{(x,b)\,;\,b\in B, x\in X_b\}$ is a compatible positive cone.

\begin{prop}
A family $(X_b)_{b\in P_B}$ of subsets of $X$ is compatible if and only if it is satisfies the following conditions:
\begin{enumerate}
\item[\em (1)] $X_b\neq \emptyset\,\Leftrightarrow\,b\in P_B\,\Leftrightarrow\,0\in X_b$;
\item[\em (2)] $X_0=P_X$;
\item[\em (3)] $(\forall b,b'\in P_B)\;X_b+\varphi_b(X_{b'})\subseteq X_{b+b'}$;
\item[\em (4)] $(\forall a\in B)\,(\forall b\in P_B)\,(\forall x\in X)\;x+\varphi_a(X_b)\subseteq X_{a+b-a}+\varphi_{a+b-a}(x)$.
\end{enumerate}
\end{prop}

\begin{proof}
(1) and (2) guarantee that $\pi_B$ is monotone and the order of $X$ is inherited from the order of $X\rtimes_\varphi B$, while (3) and (4) are equivalent to closure of $P$ under addition and conjugation, respectively.
\end{proof}

\begin{remark}
Although the previous result is just a way of formulating the closure under addition and conjugation of $P$, it gives interesting information on the behaviour of $X_b$ under the action $\varphi_a$. Indeed, one concludes that, if $(X_b)_{b\in B}$ is compatible, then:
\[(\forall a\in B)\,(\forall b\in P_B)\;\varphi_a(X_b)=X_{a+b-a}.\]
This gives that, in particular, for all $b\in P_B$, $\varphi_b(X_b)=X_b$, and that, for conjugate positive elements $b,b'$ of $B$, $X_b$ is isomorphic to $X_{b'}$, via the action $\varphi$.
\end{remark}

\section{On $\SS$-protomodularity}

The existence of different compatible orders shows that the Split Short Five Lemma fails in $\OrdGrp$, and so this category is not protomodular. One may then ask whether $\OrdGrp$ is $\SS$-protomodular (cf. \cite[Definition 3.1]{BMMS}), for a suitable class $\SS$ of split extensions, which in this context are usually called \emph{points}, due to the fact that a split epimorphism $f\colon A\to B$ in a category $\CC$, together with its splitting $s\colon B\to A$, is nothing but a morphism from the terminal object $\id\colon B\to B$ into $f\colon A\to B$ in the slice category $\CC/B$ over $B$, which we will refer to as a \emph{point over $B$}. We will denote by $\Pt(B)$ the category of points $(f\colon A\to B,\,s\colon B\to A)$ over $B$ where a morphism $h\colon (f,s)\to (f',s')$ is a morphism $h\colon A\to A'$ in $\CC$ such that $f'\cdot h=f$ and $h\cdot s=s'$; $\Pt_\SS(B)$ is its full subcategory of points in $\SS$.

We recall that a split extension
\begin{equation}\label{eq:point}
\xymatrix{ X \ar[r]^k & A \ar@<-2pt>[r]_f & B \ar@<-2pt>[l]_s}
\end{equation}
or, equivalently, a point $(f,s)$ with kernel $k$, is \emph{strong} if $k$ and $s$ are jointly strongly epimorphic. It is \emph{stably strong} if every pullback of it along any morphism $g\colon C\to B$ is a strong point.

As in every protomodular category, in $\Grp$ every point is strong, hence also stably strong, but that is not the case in $\OrdGrp$. Below we identify the strong points in $\OrdGrp$. Also, one may wonder whether they are related to the points $(f,s)$ such that $(f,s)$ is a \emph{rali} ($f$ is \emph{r}ight \emph{a}djoint and \emph{l}eft \emph{i}nverse to $s$) with respect to the $\Ord$-enrichment of $\OrdGrp$ studied in \cite{CMR}: given two morphisms $g,h\colon X\to Y$ in $\OrdGrp$, $g\leq h$ if $g(x)\leq h(x)$ for every \emph{positive} element $x$ of $X$.

\begin{lemma}
Given a point \eqref{eq:point} in $\OrdGrp$, where we identify $A$ with $X\rtimes_\varphi B$ as usual, and $P$ is its positive cone,
\begin{enumerate}
\item[\em (1)] $(f,s)$ is a rali if, and only if, $P=P_\prod$;
\item[\em (2)] $(f,s)$ is strong if, and only if, $P$ is minimal (i.e. $P$ is generated by $P_\prod$).
\end{enumerate}
\end{lemma}

\begin{proof}
(1) Assume that $(f,s)$ is a rali, i.e. $f\cdot s=\id_B$ and $s\cdot f\leq \id_A$. Then, for every $(x,b)\in P$, $(0,b)= s(f(x,b))\leq (x,b)$, and therefore $(0,0)\leq (x,0)$, which is equivalent to $x\geq 0$ since $k$ is an extremal monomorphism.

(2) Since $(f,s)$ is always strong as a point in $\Grp$, we only have to show that, if $k$ and $s$ factor through a bijective morphism, then it is an isomorphism, and this is easily seen to be the case exactly when $P$ is minimal.
\end{proof}

\begin{remark}
As a side remark we mention that a point \eqref{eq:point} is a rali exactly when the split extension $\xymatrix{P_X\ar[r]^{\langle 1,0\rangle}&P\ar@<-2pt>[r]_{\pi_2}&P_B\ar@<-2pt>[l]_{\langle 0,1\rangle}}$ is a Schreier point (see \cite{Schreier book}) in the category of monoids.
\end{remark}

In \cite[Proposition 6.2]{CMFM} it is shown that $\OrdGrp$ is $\SS$-protomodular when $\SS$ is the class of split extensions with the product order, i.e. rali points. Here we analyse whether there is a larger class $\SS$ of points that makes $\OrdGrp$ an $\SS$-protomodular category. Denoting the classes of rali, strong and stably strong points by $\Rali$, $\Strong$ and $\Strong^*$ respectively, we know that
\[\Rali\subseteq \Strong^*\subseteq\Strong\hspace*{15mm}\mbox{ and that we must have } \hspace*{15mm}\Rali\subseteq \SS\subseteq\Strong^*.\]
First we show that the two inclusions on the left are strict.

\begin{prop}
\begin{enumerate}
\item[\em (1)] Strong points are not stable under pullback.
\item[\em (2)] There is a stably strong point which is not a rali.
\end{enumerate}
\end{prop}

\begin{proof}
(1)
With $\ZZ_0$ and $\ZZ$  the group of integers, respectively with $P=\{0\}$ and with the usual order, and $\varphi\colon \ZZ_0\times\ZZ\to\ZZ$ defined by $\varphi_1(x)=-x$, consider the strong point
\[\xymatrix{\ZZ_0 \ar[r]^-{\langle 1,0\rangle} & \ZZ_0\rtimes_\varphi\ZZ \ar@<-2pt>[r]_-{\pi_2} & \ZZ; \ar@<-2pt>[l]_-{\langle 0,1\rangle}}\]
that is, $\ZZ_0\rtimes_\varphi\ZZ$ is equipped with the minimal order. This point is not stably strong, as we show next.\\

Indeed, for any strong point \eqref{eq:point} such that there exists $b\geq 0$ with $\varphi_b\not\geq\id$ but $\varphi_b^m\geq\id$ for some natural number $m\neq 0$ (here by $\varphi_b^m$ we mean $\varphi_b$ computed $m$ times), we may consider its pullback along $f\colon\ZZ\to B$ with $f(n)=n\, m\, b$:
\[\xymatrix{X\rtimes\ZZ\ar[r]\ar@<2pt>[d]^{\pi_2}&X\rtimes_\varphi B\ar@<2pt>[d]^{\pi_2}\\
\ZZ\ar@<2pt>[u]^{\langle 0,1\rangle}\ar[r]_f&B\ar@<2pt>[u]^{\langle 0,1\rangle}}\]
Then the action of the point on the left is given by $\varphi_b^m$, hence the product order is the minimal compatible order in $X\rtimes\ZZ$. However, this point is not ordered by the product order: for $x\in X$ such that $\varphi_b(x)\not\geq x$ one has, since $b\geq 0$,
\[(-x,0)+(0,b)+(x,0)+(0,(m-1)b)=(-x+\varphi_b(x), m\, b)\geq 0;\]
hence $(-x+\varphi_b(x),1)$ is positive in $X\rtimes \ZZ$ although $-x+\varphi_b(x)$ is not positive, by assumption.\\

%With $\QQ$ be the ordered group of rational numbers, consider the point
%\[\xymatrix{\QQ^2 \ar[r] & \QQ^2\rtimes_\varphi\ZZ \ar@<-2pt>[r]_-{\pi_2} & \ZZ \ar@<-2pt>[l]_-{\langle 0,1\rangle}}\]
%with $\varphi_1(x,y)=(\frac{y}{2},2x)$ for every $(x,y)\in\QQ^2$, where $\QQ^2\rtimes \ZZ$ has the minimal order (which is compatible because $\varphi_1$, and therefore $\varphi_n$ for every $n\in\ZZ$, is monotone), hence it is a strong point. Note that, for instance, $((1,-2),2)$ is positive in $\QQ^2\times_\varphi\ZZ$ since
%\[((2,2),0)+((0,0),1)-((2,2),0)+((0,0),1)=((1,-2),2).\] In the pullback of $\pi_2\colon\QQ^2\rtimes_\varphi\ZZ\to\ZZ$ along $2\times (\;)\colon\ZZ\to\ZZ$:
%\[\xymatrix{\QQ^2\times\ZZ\ar[r]\ar@<2pt>[d]^{\pi_2}&\QQ^2\rtimes_\varphi\ZZ\ar@<2pt>[d]^{\pi_2}\\
%\ZZ\ar@<2pt>[u]^{\langle 0,1\rangle}\ar[r]_{2\times(\;)}&\ZZ\ar@<2pt>[u]^{\langle 0,1\rangle}}\]
%both horizontal morphisms are extremal monomorphisms, and the action on the left is the identity since $\varphi_n$ is the identity when $n$ is an even number; moreover, its order is inherited from the order on the point on the right, hence it is not the product order, while its minimal order is the product order because the action is trivial.\\

(2) Consider now the point
\begin{equation}\label{eq:strong}
\xymatrix{\QQ \ar[r]^-{\langle 1,0\rangle} & \QQ\rtimes_\varphi\ZZ \ar@<-2pt>[r]_-{\pi_2} & \ZZ \ar@<-2pt>[l]_-{\langle 0,1\rangle}}
\end{equation}
with $\varphi_n(x)=2^nx$ (hence monotone for every $n\in\ZZ$), equipped with the minimal order, and let $g\colon A\to \ZZ$ be any morphism in $\OrdGrp$. In its pullback $\QQ\rtimes_\psi A$ as in the diagram
\[\xymatrix{\QQ\rtimes_\psi A\ar[r]\ar[d]&\QQ\rtimes_\varphi\ZZ\ar[d]\\
A\ar[r]_{g}&\ZZ}\]
the positive cone is given by $P=\{(x,a)\,;\,a\geq 0\mbox{ and }(x,g(a))\geq 0\}$, and the action $\psi$ is given by $\psi_a(x)=\varphi_{g(a)}(x)$. Let us check that $\QQ\rtimes_\psi A$ has the minimal order: if $(x,a)\geq 0$ and $g(a)=0$, then both $a$ and $x$ are positive; if $(x,a)\geq 0$ and $g(a)\neq 0$, then, for $r=\frac{x}{1-2^{g(a)}}$ one gets, by closure of $P$ under conjugation,
\[(r,0)+(0,a)-(r,0)=(r-\psi_a(r),a)=(r-\varphi_{g(a)}(r),a)=(x,a)\in P.\]
Therefore, every pullback of the point \eqref{eq:strong} is a strong point, as claimed.
\end{proof}

\begin{lemma}
The class $\Strong^*$ of stably strong points in $\OrdGrp$ is closed under finite products in the category of points.
\end{lemma}

\begin{proof}
$\Strong^*$ contains the terminal object and is closed under binary products because they commute with pullbacks.
\end{proof}

\begin{theorem}
Let $\SS=\Strong^*$ be the class of stably strong points in $\OrdGrp$. Then:
\begin{enumerate}
\item[\em (1)] The Split Short Five Lemma holds with respect to $\SS$; that is, for any commutative diagram
\begin{equation}\label{eq:SSFL}
\xymatrix{0\ar[r]&X\ar[r]^k\ar[d]^a&A\ar@<-2pt>[r]_f\ar[d]^b&B\ar@<-2pt>[l]_s\ar[d]^c\ar[r]&0\\
0\ar[r]&X'\ar[r]_{k'}&A'\ar@<-2pt>[r]_{f'}&B'\ar@<-2pt>[l]_{s'}\ar[r]&0}
\end{equation}
in the sense that $b\cdot k=k'\cdot a$, $c\cdot f=f'\cdot b$ and $b\cdot s=s'\cdot c$, if the rows are split extensions belonging to $\SS$ and $a$ and $c$ are isomorphisms, then $b$ is an isomorphism as well.
\item[\em (2)] For any morphism $h\colon Y \to B$, the change of base functor $h^*\colon \Pt_\SS(B)\to \Pt_\SS(Y)$ is conservative.
\end{enumerate}
\end{theorem}

\begin{proof}
(1) Given diagram \eqref{eq:SSFL}, we know that $b$ is an isomorphism of groups, because the Split Short Five Lemma holds in $\Grp$. Since both orders are minimal, $b$ is in fact an isomorphism in $\OrdGrp$.\\

(2) Adapting the classical proof that the change of base functor between points is conservative provided that the Split Short Five Lemma holds, it is enough to observe that the change of base functor between points restricts to $\Pt_\SS$ since $\SS$ is pullback stable.
\end{proof}

It is an open problem to know whether $\Strong^*$ is stable under equalizers in the category of points, and consequently whether $\OrdGrp$ is $\Strong^*$-protomodular in the sense of \cite[Definition 8.1.1]{Schreier book}, \cite[Definition 3.1]{BMMS}, although the previous Theorem shows that $\OrdGrp$ has the desired properties for relative protomodularity with respect to $\Strong^*$. Moreover, $\OrdGrp$ is $\Strong^*$-protomodular in the sense of \cite[Definition 8.5]{B15}, where the author only imposes that $\SS$ is a stable class of strong points, and therefore this weaker notion does not assure that the change of base functor, restricted to $\SS$, is conservative (hence neither the Split Short Five Lemma). To assure that the change of base functor, restricted to $\Pt_\SS$, is conservative, the authors of \cite{BMMS} impose that $\SS$ is stable under equalizers in the category of points. However this property seems quite complicated to check, as it is the case of our example and, for instance, of the class of points studied in \cite[7.14]{MRVdL}.

Henceforth, we propose the following definition, which is in fact a translation of the notion of (absolute) protomodularity \cite[Definition 3.1.3]{BBbook}, not making use of equalizers but instead focussing on the key property of protomodularity.

\begin{definition}
If $\SS$ is a class of strong points of the category $\CC$, $\CC$ is said to be \emph{protomodular with respect to $\SS$} if
\begin{enumerate}
\item $\CC$ has pullbacks of points in $\SS$ along any morphism, which belong also to $\SS$.
\item For any morphism $h\colon Y\to B$, the change of base functor $h^*\colon\Pt_\SS(B)\to\Pt_\SS(Y)$ reflects isomorphisms.
\end{enumerate}
\end{definition}

Then we can compare these notions using Proposition 3.2 of \cite{BMMS}:

\begin{theorem}
If $\CC$ is $\SS$-protomodular in the sense of \cite{BMMS}, then the change of base functor is conservative when restricted to $\Pt_\SS$, and so $\CC$ is protomodular with respect to $\SS$.
\end{theorem}

\section{On the existence of $\SS$-classifiers}

Unlike the category $\TopGrp$ of topological groups (see \cite{CC}), $\OrdGrp$ has no split extension classifiers, as we show in Theorem \ref{th:noclass}.
Still, it is interesting to analyse the existence, in $\OrdGrp$, of split extension classifiers for \emph{special classes of points}, as we discuss in this last section.

Given a class $\SS$ of split extensions in $\OrdGrp$, we denote by $\SS_X$ the category of split extensions in $\SS$ with kernel $X$ with morphisms triples $(a,b,c)$ as in \eqref{eq:morph} with $a=\id$.

\begin{definition}
If $\SS$ is a class of split extensions in $\OrdGrp$, we say that \emph{$\OrdGrp$ has $\SS$-classifiers} if the category $\SS_X$ has a terminal object; that is, for every ordered group $X$ there exists a split extension with kernel $X$
\begin{equation}\label{eq:Sclass}
\xymatrix{X\ar[r]^-{\langle 1,0\rangle}&X\rtimes A(X)\ar@<-2pt>[r]_-{\pi_2}&A(X)\ar@<-2pt>[l]_-{\langle 0,1\rangle}}
\end{equation}
in $\SS$ such that, for each split extension in $\SS$ with kernel $X$ there exists exactly one morphism in $\SS_X$ from it into \eqref{eq:Sclass}.
\end{definition}

 Given an ordered group $X$, let $\Aut_P(X)$ be the group \[\Aut(X)=\{\alpha\colon X\to X\,;\,\alpha\mbox{ is a monotone automorphism}\}\]
  equipped with an order with positive cone $P$. Then, by Theorem \ref{th:lexic}, there is an order in $X\rtimes\Aut(X)$ making
\begin{equation}\label{eq:Aut}
\xymatrix{X\ar[r]^-{\langle 1,0\rangle}&X\rtimes \Aut_P(X)\ar@<-2pt>[r]_-{\pi_2}&\Aut_P(X)\ar@<-2pt>[l]_-{\langle 0,1\rangle}}
\end{equation} a split extension if, and only if, when $\alpha\sim\id$ in $\Aut_P(X)$ also $\alpha\sim\id$ pointwise in $X$. We call such orders in $\Aut(X)$ \emph{admissible}.

\begin{prop}
Let $\SS$ be a class of split extensions, $X$ an ordered group and $P$ an admissible positive cone in $\Aut(X)$ such that \eqref{eq:Aut} belongs to $\SS$. If
\[\xymatrix{X\ar[r]^-{\langle 1,0\rangle}&X\rtimes A(X)\ar@<-2pt>[r]_-{\pi_2}&A(X)\ar@<-2pt>[l]_-{\langle 0,1\rangle}}\]
is a classifier for $\SS_X$, then $A(X)$ is isomorphic, as a group, to $\Aut(X)$, and its positive cone $P_A$ contains $P$.
\end{prop}

\begin{proof}
By definition of classifier, there exists a unique morphism $\gamma$ making the following diagram commute
\[\xymatrix{X\ar@{=}[d]\ar[r]^-{\langle 1,0\rangle}& X\rtimes \Aut_P(X)\ar[d]\ar@<-2pt>[r]_-{\pi_2}& \Aut_P(X)\ar[d]^\gamma\ar@<-2pt>[l]_-{\langle 0,1\rangle}\\
X\ar[r]^-{\langle 1,0\rangle}&X\rtimes A(X)\ar@<-2pt>[r]_-{\pi_2}&A(X).\ar@<-2pt>[l]_-{\langle 0,1\rangle}}\]
On the other hand, being a split extension in $\Grp$, there is a group homomorphism $\psi\colon A(X)\to\AUT(X)$ as in diagram \eqref{eq:AutX}. Since we are in $\OrdGrp$, the image of $\psi$ has only monotone automorphisms, i.e. $\psi$ factors through $\Aut(X)$ as $\overline{\psi}\colon A(X)\to \Aut(X)$. Clearly $\overline{\psi}\cdot\gamma=\id$ and $\gamma\cdot\overline{\psi}=\id$, hence we may assume that $A(X)$, as a group, is $\Aut(X)$; moreover, from the monotonicity of $\gamma$ it follows that its positive cone $P_A$ contains the positive cone $P$ of $\Aut_P(X)$.
\end{proof}

%Given an ordered group $X$, let $\Aut^\sim(X)$ be the group \[\Aut(X)=\{\alpha\colon X\to X\,;\,\alpha\mbox{ is a monotone automorphism}\}\] equipped with the order $\leq$ defined by $\alpha\leq\alpha'$ if $\alpha(x)\leq\alpha'(x)$ for every $x\in X$. (We point out that this is in fact an equivalence relation, with positive cone $P^\sim=\{\alpha\in \Aut(X)\,;\,\alpha\sim\id\}$.)

\begin{theorem}
For every ordered group $X$, let $\widetilde{P}=\{\alpha\,;\,\alpha(x)\sim x \mbox{ for all }x\in X\}$. The rali point
\begin{equation}\label{eq:*}
\xymatrix{X\ar[r]^-{\langle 1,0\rangle}&X\rtimes \Aut_{\widetilde{P}}(X)\ar@<-2pt>[r]_-{\pi_2}&\Aut_{\widetilde{P}}(X).\ar@<-2pt>[l]_-{\langle 0,1\rangle}}
\end{equation}
is a terminal object of $\Rali_X$.
\end{theorem}

\begin{proof}
First of all it is easy to check that the product order in $X\rtimes \Aut_{\widetilde{P}}(X)$, which in fact coincides with the lexicographic order, is compatible in \eqref{eq:*}. So \eqref{eq:*} is a rali point.

Given a rali point $\xymatrix{X\rtimes_\varphi B\ar@<-2pt>[r]_-{\pi_2}&B\ar@<-2pt>[l]_-{\langle 0,1\rangle}}$ with kernel $X$, we know that there is a unique group homomorphism $\overline{\varphi}\colon B\to\AUT(X)$, which factors through $\Aut(X)$, making diagram \eqref{eq:AutX} commute. Moreover, its corestriction $\widetilde{\varphi}\colon B\to \Aut_P(X)$ is monotone -- since $X\rtimes_\varphi B$ has the product order, by Proposition \ref{prop:prod}(2) $\varphi\sim\id$ for every $b\geq 0$ -- and $1\times\widetilde{\varphi}\colon X\rtimes_\varphi B\to X\rtimes\Aut^\sim(X)$ is also clearly monotone. Uniqueness of $\overline{\varphi}$ guarantees uniqueness of $\widetilde{\varphi}$.
\end{proof}

\begin{corollary}
$\OrdGrp$ has $\Rali$-classifiers. \hfill$\Box$
\end{corollary}

This result can be naturally extended in the following way.

\begin{prop}
Let $X$ be an ordered group and $P$ an admissible positive cone in $\Aut(X)$. Then the maximal point
\[\xymatrix{X\ar[r]^-{\langle 1,0\rangle}&X\rtimes \Aut_P(X)\ar@<-2pt>[r]_-{\pi_2}&\Aut_P(X)\ar@<-2pt>[l]_-{\langle 0,1\rangle}}\]
classifies the class $\SS_X$ of split extensions $\xymatrix{X\rtimes_\varphi B\ar@<-2pt>[r]_-{\pi_2}&B\ar@<-2pt>[l]_-{\langle 0,1\rangle}}$ with kernel $X$ such that
\begin{enumerate}
\item[\em (1)] for all $b\in P_B$, $\varphi_b\in P$;
\item[\em (2)] for all $x\in X$, $x\geq 0$ provided that there exists $b\in P_B$ with $(x,b)\geq 0$ in $X\rtimes_\varphi B$ and $\varphi_b\sim\id$ in $\Aut_P(X)$.
\end{enumerate}
\end{prop}

\begin{proof}
Given a point in $\SS_X$, we know that there is a unique group homomorphism $\overline{\varphi}\colon B\to \Aut(X)$ making the following diagram
\[
\xymatrix{X\ar@{=}[d]\ar[r]^-{\langle 1,0\rangle}&X\rtimes_\varphi B\ar@<-2pt>[r]_-{\pi_2}\ar[d]^{1\times\overline{\varphi}}&B\ar@<-2pt>[l]_-{\langle 0,1\rangle}\ar[d]^{\overline{\varphi}}\\
X\ar[r]^-{\langle 1,0\rangle}&X\rtimes\Aut_P(X)\ar@<-2pt>[r]_-{\pi_2}&\Aut_P(X)\ar@<-2pt>[l]_-{\langle 0,1\rangle}}\]
commute. Then condition (1) guarantees that $\overline{\varphi}$ is monotone, and (2) gives monotonicity of $1\times\overline{\varphi}$.
\end{proof}

\begin{remark}
Given an ordered group $X$ such that, for every $x\in X$, either $x\geq 0$ or $-x\geq 0$, consider in $\Aut(X)$ the admissible order defined by $P^+=\{\alpha\,;\,(\forall x\in P_X)\;\alpha(x)\geq x\}$. We denote by $\Aut^+(X)$ this ordered group and by $\SS^+$ the class of split extensions defined as in the proposition above, that is, $\SS^+$ is the union of $\SS^+_X$ for all such ordered groups $X$. The split extensions classified by $\Aut^+(X)$ depend very much on the action $\varphi$. For instance, $\Aut^+(\QQ)$ classifies a split extension
\[\xymatrix{\QQ\ar[r]^-{\langle 1,0\rangle}&\QQ\rtimes_\varphi \ZZ\ar@<-2pt>[r]_-{\pi_2}&\ZZ\ar@<-2pt>[l]_-{\langle 0,1\rangle}}\]
when $\varphi=\id$ only if it is a rali, while when $\varphi_n(x)=2^n\,x$ it classifies any such split extension.\\

We could use analogously the order in $\Aut(X)$ defined by $P^-=\{\alpha\,;\,(\forall x\leq 0)\; \alpha(x)\geq x\}$.
\end{remark}

\begin{theorem}\label{th:noclass}
$\OrdGrp$ has no split extensions classifiers.
\end{theorem}

\begin{proof}
For each ordered group $X$, assume that $\xymatrix{X\rtimes \Aut_P(X)\ar@<-2pt>[r]_-{\pi_2}&\Aut_P(X)\ar@<-2pt>[l]_-{\langle 0,1\rangle}}$ classifies the split extensions with kernel $X$.
As we remarked above, when $X$ is an ordered group such that each $x\in X$ is either positive or negative, both $P^+$ and $P^-$ are admissible positive cones in $\Aut(X)$. Hence, both $P^+$ and $P^-$ are contained in $P$. Since $\alpha\in P^+$ if and only if its inverse $\alpha^{-1}\in P^-$, this shows that if $\alpha\in P^+$ then $\alpha\sim\id$ in $P$, and so $\alpha\sim\id$ pointwise in $X$ because $P$ must be admissible. But this is not true in general: for instance, for $X=\QQ$ and $n$ a natural number larger than $1$, $\alpha(x)=n\,x\in P^+$ but pointwise $\alpha\not\sim\id$.
\end{proof}

\section*{Acknowledgements}

The results of this paper were obtained while the first author was a tutor of the second author under the programme \emph{Novos Talentos em Matem\'{a}tica} from the Gulbenkian Foundation.

We thank Graham Manuell for valuable discussions on the subject of this paper. We also thank Diana Rodelo, Andrea Montoli and Manuela Sobral for useful discussions on the concept of $\SS$-protomodularity.

\end{document}